\documentclass[letterpaper, 10 pt, conference]{ieeeconf}  

\usepackage{amsmath}
\usepackage{amssymb}

\usepackage{nicefrac}
\usepackage{mathtools}

\usepackage[usenames,dvipsnames,svgnames,table]{xcolor}

\usepackage{amsfonts}
\usepackage{tikz}
\usetikzlibrary{calc,shapes,arrows,automata}

\newtheorem{theorem}{Theorem}
\newtheorem{lemma}{Lemma}

\newtheorem{definition}{Definition}
\newtheorem{remark}{Remark}

\usepackage{thmtools}
\declaretheorem{example}
\renewcommand\thmcontinues[1]{Continued}

\usepackage{xcolor}
\usepackage{hyperref}

\hypersetup{hidelinks,linkcolor=black}

\newcommand{\R}{\mathbb{R}}
\newcommand{\N}{\mathbb{N}}
\newcommand{\Z}{\mathbb{Z}}

\newcommand{\A}{\mathcal{A}}

\newcommand{\cR}{\mathcal{R}}
\newcommand{\cH}{\mathcal{H}}
\newcommand{\cZ}{\mathcal{Z}}
\DeclareMathAlphabet\mathbfcal{OMS}{cmsy}{b}{n}

\newcommand{\cG}{\mathcal{G}}
\newcommand{\cB}{\mathcal{B}}

\newcommand{\CC}{\mathcal{C}}

\newcommand{\T}{\mathcal{T}}

\newcommand{\cN}{\mathcal{N}}

\renewcommand{\emptyset}{{\varnothing}}

\newcommand{\twocol}[1]{}

\newcommand{\no}{\raisebox{0.13\baselineskip}{\ensuremath{{\scriptstyle \#}}}}
\newcommand{\intcc}[1]{\ensuremath{{\left[#1\right]}}}

\newcommand{\qed}{\hfill \QED}
\usepackage{soul}	
\definecolor{hlcolor}{rgb}{.5,.5,1}	
\sethlcolor{hlcolor}
\setulcolor{red}

\usepackage{calc}

\newlength\dlf

\IEEEoverridecommandlockouts                              

\title{\LARGE \bf
On a notion of entropy for reachability properties
}

\author{Mahendra Singh Tomar and Majid Zamani
\thanks{*This work was supported by the NSF under Grant CMMI-2013969.}
\thanks{Mahendra Singh Tomar and Majid Zamani are with the Computer Science Department, 
        University of Colorado Boulder, USA. Majid Zamani is also with the Institute of Informatics, LMU Munich, Germany.
        {\tt\small mahendra.tomar@colorado.edu, majid.zamani@colorado.edu}}%
}

\begin{document}

\maketitle
\thispagestyle{empty}
\pagestyle{empty}

\begin{abstract}

In this work, we introduce a notion of reachability entropy to characterize the smallest data rate which is sufficient enough to enforce reach-while-stay specification. We also define data rates of coder-controllers that can enforce this specification in finite time. Then, we establish the data-rate theorem which states that the reachability entropy is a tight lower bound of the data rates that allow satisfaction of the reach-while-stay specification. For a system which is related to another system under feedback refinement relation, we show that the entropy of the former will not be larger than that of the latter. We also provide a procedure to numerically compute an upper bound of the reachability entropy for discrete-time control systems by leveraging their finite abstractions. Finally, we present some examples to demonstrate the effectiveness of the proposed results.

\end{abstract}


\section{INTRODUCTION}
	With the advancement in technology and production, 
	sensors and actuators are becoming smaller and less expensive. This promotes their extensive deployment in many real-life applications. 	
	Spatially distributed components including  plants, sensors, controllers, and actuators often exchange information over a  communication network, and they all together constitute networked control systems (NCS). Compared to classical control systems that involve point-to-point wiring between the sensors and the controllers and direct connection between them, NCS offer many advantages such as reduced wiring, cost efficiency, greater system flexibility and ease of modification. 
	NCS find applications in many areas such as traffic networks, power grids, manufacturing plants, automobiles and transportation networks.
	
	Unfortunately, the use of communication networks in feedback control loops makes the analysis and design of NCS much more complex. In NCS, the use of digital channels for data transfer from the sensors to controllers, limit the amount of information that can be transferred per unit of time, due to the finite data rate of the channel. This introduces quantization errors that can adversely affect the control performance. The problem of control and state estimation over a digital communication channel with a limited bit-rate has attracted a lot of attention in the last two decades; see for example \cite[and references therein]{hespanha2002towards,nair2003exponential,tatikonda2004control}.
	A survey of results on data-rate-limited control can be found in \cite{NairFagniniZampieriEvans07}. 

	For efficient utilization of the network resources, it becomes increasingly important that our feedback loops operate at the smallest permissible data rates. 
	A tight lower bound on the data rate of a digital channel between the coder and controller, to achieve some control task like stabilization or invariance, can be characterized in terms of some notions of \emph{entropy}  which is described as an intrinsic property of the system and is independent of the choice of the coder-controller.

In~\cite{NairEvansMarrelsMoran04}, the authors introduced a notion of topological feedback entropy as a measure of the rate at which a discrete-time control system generates information, with states confined in a given compact set. They established that the entropy is equal to the smallest average bit rate that allows to enforce invariance of a subset of the state space, and thus the channel must transfer information at a rate faster than the rate of information generation. They defined the entropy based on open covers, with each cover element assigned a finite length open loop control sequence. Later, the results in \cite{ColoniusKawan09} introduced a notion of invariance entropy for continuous-time control systems based on the minimum cardinality amongst all sets of control functions that can make the desired set invariant. For networks of control subsystems, the results in~\cite{kawan2015network} introduced a notion of subsystem invariance entropy to characterize the smallest data rate, from a centralized controller to the subsystems, needed to make a subset of the state space invariant. 

For discrete-time dynamical systems, the results in~\cite{savkin2006analysis} established that the channel data rate not less than the topological entropy is needed for state observation.
The notion of estimation entropy was introduced in~\cite{liberzon2017entropy} to describe the critical data rate for state estimation with a given exponential convergence rate. 
For networked systems, the results in~\cite{matveev2019comprehending} related the topological entropy with the minimal data rates for observation of the network's state.
Restoration entropy was introduced in~\cite{PartII} to describe the minimal data rate for state estimation above which the estimation quality  can also be exponentially improved.
The results in~\cite{sibai2017optimal} extended the notion of estimation entropy to the case of switched nonlinear dynamical systems and related it to the minimal data rate needed for state estimation with an error that decays exponentially but only after a specified period of time after each switch. For switched linear systems, the relation between the topological entropy as defined in~\cite{yang2018topological} and global exponential stability wass analyzed in~\cite{yang2019topological}.
For switched linear systems with arbitrary switching, the work in~\cite{berger2020worst} introduced a notion of worst-case topological entropy to describe the minimal data rate required for state observation with exponentially decreasing estimation error.

In this work, we focus on the setting depicted in Figure~\ref{f:1}, in which the state is encoded at the sensor side and transmitted over a noiseless digital channel to the controller located near the actuator.
We are interested in characterizing the minimal data rate of the digital channel in the feedback loop that permits the satisfaction of a desired control task. Given sets $T\subset Q\subset X$ and any initial state in $Q$, the task is to reach $T$ without ever leaving $Q$.
Here, for the first time, we describe the $(Q,T)$-reachability entropy and the associated definitions. Our definition of the reachability entropy is inspired from the proposed framework in~\cite{rungger2017invariance} and is based on the set of sequences of cover elements. We present the definition of the coder-controllers of interest and their data rates. We show that the smallest data rate amongst the coder-controllers that can enforce the reachability specification is given by the $(Q,T)$-reachability entropy. In addition, we relate the entropies of two systems related under a feedback refinement relation introduced in \cite{reissig2016feedback}. In particular, given a continuous-space control system and its finite abstraction, if the two systems are related under a feedback refinement relation, then an upper bound of the entropy of the original system can be computed using that of its finite abstraction. We describe a procedure to numerically compute an upper bound of the reachability entropy through construction of a graph, the details of which are provided in the description of Example~\ref{eg:rommtemp}. Finally, we elaborate on the results using three case studies.

%

\begin{figure}[t]
	\centering
	\begin{tikzpicture}[node distance=1.8cm, thick, >=latex]
		
		\tikzset { 
			block/.style={ 
				draw,
				thick, 
				rectangle, 
				minimum height = .8cm, 
				minimum width = 2cm},
		}
		\draw node at (0,0) [block] (F)  {System};
		\draw node at (3,-1) [block] (H)  {Sensor/Coder};
		\draw node at (-3,-1) [block] (C)  {Controller};
		
		\draw[->] (F.east)  -|  (H.north);
		\draw[->] (C.north)  |- (F.west);
		\draw[->] (H.south)  |-       (0,-2)        node[above] {\small digital
			noiseless channel $R$ bits/time unit}  -|          (C.south);
	\end{tikzpicture}
	\caption{Coder-controller feedback loop.}\label{f:1}
	\vspace{-0.5cm}
\end{figure}
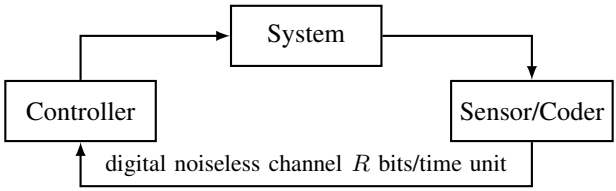

\emph{\textbf{Notation.}}
For a set $A$, we use $\no A$ to denote 
its cardinality.
{By $[k_1;k_2]$ we refer to the set of integers $\{j\mid k_1\leq j\leq k_2\}$.}
We use $f:A\rightrightarrows B$ to denote a set-valued map from $A$ to $B$. If $f$ is set-valued, then $f$ is \emph{strict} if for every $a\in A$ we have $f(a)\neq \emptyset$. The restriction of $f$ to a subset $M\subseteq A$ is denoted by $f|_M$. We use $B^A$ to denote the set of all functions $f:A\to B$. For a relation $R\subseteq A\times B$ and $D\subseteq A$, we define $R(D):=\cup_{d\in D}R(d)$. The concatenation of two functions $x:[0;a)\to X$ and $y:[0;b)\to X$ with $a,b\in\N$ is denoted by $xy$ and defined as $xy(t):=x(t)$ for $t\in[0;a)$ and $xy(t):=y(t-a)$ for $t\in[a;a+b)$. A \emph{cover} $\A$ of a set $B$, is a set of subsets of $B$ such that $B=\cup_{A\in \A}A$.

\section{Reachability Entropy}\label{sec:pre}

We define a \emph{system} as a triple 
\begin{equation}\label{e:sys}
	\Sigma:=(X,U,F),
\end{equation}
where $X$ and $U$ are
nonempty sets and $F:X\times U\rightrightarrows X$ is the transition function.
%
%




Given sets $T$ and $Q$, with $T\subset Q\subset X$, we are interested in coder-controllers that can steer any state in $Q$ to the set $T$ in finite time.
If this objective can be achieved, we say that $(Q,T)$-reachability is satisfiable for the system. 

\begin{definition}[$(Q,T)$-reachability]\label{def:QTreachability}
	$(Q,T)$-{reachability} is said to be satisfiable for $\Sigma$, if every state in $Q$ can be taken to the set $T$ in finite time using a finite set of control inputs, while never leaving $Q$. 
\end{definition}

Next we describe reach-spanning sets which are defined in terms of the sequences of cover elements.
Consider a tuple $(\A,G)$ 
where $\A$ is a finite cover of $Q\backslash T$, and a map $G:\cup_{t\in\Z_{\geq 0}}\A^{[0;t]}\to U$. 
%
%
%
Further consider a set $\cR = \{\alpha_1,\ldots,\alpha_N\}$ where $N\in\N$ and $\alpha_i$ is a sequence of length $\tau_i\in \N$ in the power set of $Q$.



 \begin{definition}[$(Q,T)$-reach-spanning set]\label{def:reachSpanning}
	A set $\cR= \{ \alpha_1, \ldots, \alpha_N\}$ is called $(Q,T)$-reach-spanning in $(\A,G)$ if it satisfies the following three conditions:
\begin{enumerate}
	\item $\{\alpha_i(0)\mid \alpha_i\in\cR\}$ covers $Q$, i.e., $Q = \cup_{i\in\intcc{1;N}}\alpha_i(0)$,
	\item\label{def:reachSpanning:finalElement} only the last element of every sequence is equal to $T$  while all other elements of the sequence do not intersect the target set, i.e., $\alpha_i|_{[0;\tau_i-2]}\in \A^\intcc{0;\tau_i-2}$, $\alpha_i(\tau_i-1)= T$,
	\item \label{def:reachSpanning:successorsCoverPost} for every sequence $\alpha_i$ and $t\in\intcc{0;\tau_i-2}$, the following holds
	\begin{equation*}
		F(\alpha_i(t),G(\alpha_i|_{[0;t]})) \subseteq \cup_{A\in P_\cR(\alpha_i|_{\intcc{0;t}})}A,
	\end{equation*}
	where 
	\begin{multline*}
		P_\cR(\alpha_i|_{\intcc{0;t}}) := \{A\in \A\cup\{T\}\mid \\ \exists_{\alpha_k\in\cR}  \alpha_k|_{\intcc{0;t}}=\alpha_i|_{\intcc{0;t}}  \wedge  A=\alpha_k(t+1) \}.
	\end{multline*}
\end{enumerate}
\end{definition}

{Consider a trivial case of the  $(Q,T)$-reachability satisfaction with a single control value.
For $\A_1:= \{A\}$, $A:= Q\backslash T$ and a fixed $u_1\in U$, consider a map 
${G_1}(\alpha):=u_1$
for all $\alpha\in \cup_{t\in\Z_{\geq 0}}\A_1^{[0;t]}$.
We also call a set 
$\cR := \{T\}$
to be $(Q,T)$-reach-spanning in $(\A_1,G_1)$ if every state in $Q\backslash T$ can be taken to $T$ with the control input $u_1$ in finite time, more precisely, 
for any $x\in Q\backslash T$ for any trajectory\footnote{By trajectory $\xi_x$, we refer to any sequence of states such that $\xi_x(0) = x$ and $\xi_x(t)\in F(\xi_x(t-1),u_1)$ $\forall t \geq 1$. } $\xi_{x}$ there exists  $n_{\xi_x}\in \mathbb{N}$ such that $\xi_x(n_{\xi_x})\in T$ and $\xi_x(t)\in Q\backslash T$, for $t\in [0;n_{\xi_x}-1]$.

}

	Note that, for $\alpha\in\cR$, $P_\cR(\alpha|_{\intcc{0;t}})$ gives the set of the successors of the sequence $\alpha|_{\intcc{0;t}} $ in the set $\cR$.
	Similar to $P_\cR$, we also define $\hat{P}_\cR$ which gives the set of the successors excluding the target set, i.e.,
	\begin{equation*}
		\hat{P}_\cR(\alpha_i|_{\intcc{0;t}}) := P_\cR(\alpha_i|_{\intcc{0;t}}) \backslash T.
	\end{equation*}

	The following lemma relates the satisfaction of $(Q,T)$-reachability with the existence of a $(Q,T)$-reach-spanning set. 
	\begin{lemma}\label{lem:QTreachabilityReachSpanningSet}
		For a system $\Sigma$, $(Q,T)$-reachability is satisfiable if and only if there exists a $(Q,T)$-reach-spanning set $\cR$ in some $(\A,G)$, where $\A$ is a finite cover of $Q\backslash T$ and $G:\cup_{t\in\Z_{\geq 0}}\A^{[0;t]}\to U$ is a map.
	\end{lemma}
	\begin{proof}
		From Definition~\ref{def:QTreachability}, the satisfaction of $(Q,T)$-reachability implies that there exists a finite subset $\tilde{U}$ of $U$ such that every state in $Q$ can be taken to $T$ within some finite time $\tau_m$. The finiteness of $\tilde{U}$ and $\tau_m$ clearly implies existence of a $(Q,T)$-reach-spanning set $\cR$ in some $(\A,G)$.

		Now, consider a $(Q,T)$-reach-spanning set $\cR${, $\no \cR>1$} in some $(\A,G)$. 
		By the definition of $\cR$, for every $x_0\in Q$ there exists an $\alpha\in\cR$ such that $x\in \alpha(0)$. From condition~\ref{def:reachSpanning:successorsCoverPost} in Definition~\ref{def:reachSpanning}, we have that the next state $x_1$ under the control input $G(\alpha(0))$ lies in one of the successor cover element of $\alpha(0)$, i.e., $x_1\in A$ for some $A \in  P_\cR(\alpha|_{\intcc{0;0}})$. From the condition~\ref{def:reachSpanning:finalElement} in Definition~\ref{def:reachSpanning} we have that the last element of $\alpha$ equals $T$. Since every $\alpha\in\cR$ is of finite length, the state will eventually reach $T$ in finite time, without ever leaving $Q$.
	\end{proof} 

We define $\cR_0:=\{\alpha(0)\mid \alpha\in\cR, \alpha(0)\neq T\}$ to denote the set of first elements of the members of the set $\cR$ excluding the target set $T$.
For $\alpha_i\in \cR$ and $\tau_i=$ length$(\alpha_i)$, let
\begin{equation}\label{eq:Balpha}
	\cB(\alpha_i):= 
		\frac{1}{\tau_i-1}\big[ 	\sum_{t=0}^{\tau_i-3}\log_2\no \hat{P}_\cR(\alpha_i|_{\intcc{0;t}})+\log_2 \no\cR_0\big],  	
\end{equation}
when $\tau_i>1$  while  $\cB(\alpha_i):= 0$ for $\tau_i=1$.


%
We use $\cN(\cR)$ to denote the largest value assigned by $\cB$ to any $\alpha\in \cR$, i.e.
\begin{equation}\label{eq:NofR}
	\cN(\cR):= \max_{\alpha\in\cR}\cB(\alpha).
\end{equation}
Note that, if there exists $\alpha_k\in\cR$ with length $\tau_k=2$ then
\begin{align*}
	 \cB(\alpha_k) = \log_2 \no\cR_0.
\end{align*}
For any $\alpha_i\in\cR$, $\cB(\alpha_i)$ can be larger than $\log_2 \no\cR_0$ if $\no \hat{P}_\cR(\alpha_i|_{\intcc{0;t}}) >  \no\cR_0$ for some $t\in\intcc{0;\tau_i-{3}}$. If this is not the case, i.e., for every $\alpha_i\in \cR$, $t\in\intcc{0;\tau_i-{3}}$ we have $\no \hat{P}_\cR(\alpha_i|_{\intcc{0;t}}) <  \no\cR_0$, and if there exists an $\alpha_k\in\cR$ with $\tau_k=2$, then $\cN(\cR) = \log_2 \no\cR_0$.

Let us define:
\begin{multline*}
	r(\A,G,Q,T):= \min\{\cN(\cR)\mid \cR \text{ is }\\
	\text{$(Q,T)$-reach-spanning }  \text{set in } (\A,G) \}.
\end{multline*}
Then the $(Q,T)$-\emph{reachability entropy} is defined as 
\begin{equation*}
	h(Q,T):= {\inf_{(\A,G)}} r(\A,G,Q,T).
\end{equation*}

If $(Q,T)$-reachability is satisfiable for $\Sigma$ then $h(Q,T)$ is finite, otherwise it is defined to be infinite.



\section{Admissible Coder-Controller}\label{sec:coder-controller}
For a given system $\Sigma$, we focus on the set $\mathcal{H}(Q,T)$ of all coder-controllers that can take every state in $Q$ to $T$ in {finite time}, without ever leaving $Q$.
We consider the setup shown in Figure~\ref{f:1} where  a coder, located at the sensor side,
encodes the current state of the system using the
finite \emph{coding alphabet} $\bar{S} =S\cup \{s^\emptyset\}$. It transmits a symbol $s_t\in {S}$ via a digital 
noiseless channel to the controller.
The transmitted symbol $s_t$ might depend on all past states and is
determined by the \emph{coder function}
\begin{equation*}
	\textstyle
	\gamma:\bigcup_{t\in\Z_{\ge0}}Q^{\intcc{0;t}}\to \bar{S}.
\end{equation*}
For $\xi\in Q^{\intcc{0;t}}, t\in\Z_{\geq 0}$, if $\xi(k) \in T$ then $\gamma(\xi|_{\intcc{0;\hat{t}}}) := s^{\emptyset}$, for all $k\leq \hat{t}\leq t$. 
{When coder generates $s^\emptyset$, then no symbol is transmitted over the channel. }

Let $\mathcal{Z}$ denote the set of all possible finite-length symbol sequences generated in the closed loop by the coder-controller for initial states in $Q$. 
At time $t\in\Z_{\ge0}$, the controller received $t+1$ symbols $s_0\ldots s_t$, which
are used to determine the control input given by the
\emph{controller function} 
\begin{equation*}
	\textstyle
	\delta:\mathcal{Z}\to U.
\end{equation*}
When no symbol transmitted over the channel, a fixed input is applied by the controller to the system.


A $(Q,T)$ \emph{admissible coder-controller} for~\eqref{e:sys} is a triple $H:=(\bar{S},\gamma,\delta)$,
where $H\in\mathcal{H}(Q,T)$, $\bar{S}$ is a coding alphabet, and $\gamma$ and $\delta$ are compatible coder and controller function, respectively. 




Let $\hat{\mathcal{Z}}:= \{\omega s^\emptyset\mid \omega\in \mathcal{Z}\}$, i.e., elements of $\hat{\mathcal{Z}}$ are constructed by concatenation of the symbol $s^\emptyset$ with the elements of the set $\mathcal{Z}$. For $\zeta \in \hat{\mathcal{Z}}$, if $\tau$ is the length of $\zeta$, then for all $j\in\intcc{0;\tau-2}$, we have $\zeta(j)\in S$ and $\zeta(\tau-1)=s^\emptyset$. 
Note that $s^\emptyset$ is not transmitted over the channel, and is considered only to mark the end of a symbol sequence.

Given a symbol sequence $\zeta \in \hat{\mathcal{Z}}$, by $Z(\zeta|_{[0;t]})$ we denote the set of the possible successor coder symbols $s\in S$ of the symbol sequence
$\zeta|_{[0;t]}$ in the closed loop.
For $\zeta\in \hat{\mathcal{Z}}$ of length $\tau$ and $t \in [0;\tau-2]$, $Z(\zeta|_{[0;t]}) := \{s\in S\mid \zeta|_{[0;t]}s\in \hat{\mathcal{Z}} \}$. 


For notational convenience,  let us use the convention $Z(\emptyset):= \{\zeta(0) \mid \zeta \in \mathcal{Z} \} $, to account for $Z(\zeta|_{[0;0)})$ for any
sequence $\zeta$ in $\hat{\mathcal{Z}} $. 

We define the \emph{transmission data rate} of a coder-controller $H\in\mathcal{H}(Q,T)$ by
\begin{equation}\label{eq:datarate}
	R(H) := \max_{
			\zeta\in\hat{\mathcal{Z}}, {\tau>1}} \frac{1}{\tau-1}\sum_{t=0}^{\tau-2}\log_2 \no Z(\zeta|_{[0;t)}),
\end{equation}
where $\tau$ is the length of the symbol sequence $\zeta$.
{For the trivial coder-controller $H\in \cH(Q,T)$ that does not involve any symbol transmission $(\hat{\mathcal{Z}}=\{s^\emptyset\})$, 
we define $R(H) = 0$.
}


\begin{theorem}[data-rate theorem]\label{thm:dataRateThm}
	The $(Q,T)$-reachability entropy equals the smallest average data rate amongst the coder-controllers that can take every state in $Q$ to $T$, while never leaving $Q$, in finite time, i.e.,
	\begin{equation}
		h(Q,T) = {\inf_{H\in \cH(Q,T)}} R(H).
	\end{equation}
\end{theorem}
\begin{proof}
    \emph{Showing $h(Q,T)\leq {
    \inf_{H\in \cH(Q,T)} R(H)}$}:
	{for a fixed $\varepsilon>0$ we pick a coder-controller $H = (\bar{S},\gamma,\delta)\in \cH(Q,T)$ such that $R(H) \leq \inf_{\tilde{H}\in \cH(Q,T)} R(\tilde{H}) + \varepsilon$.
	We assume $H$ to be non-trivial, i.e., $R(H)>0$.    }
	{First we construct a $(Q,T)$-reach-spanning set from $H$.} 
	Consider $\cZ$, the set of all possible finite-length symbol sequences generated in the closed loop, for initial states in $Q$. From $\cZ$ we construct a finite cover of $Q\backslash T$. For $\zeta\in\cZ$ of length $\tau$ and $j\in[0;\tau-1]$, we define the set 
	\begin{multline}
	A({\zeta|_{[0;j]}}):= \{x_j\in Q\backslash T\mid \exists x_i\in A({\zeta|_{[0;i]}}) \\ \text{ for all } i\in[0;j-1] \wedge \gamma((x_i)_{i=0}^j)=\zeta(j)\}
	\end{multline}
	and assign $G((A({\zeta|_{[0;i]}}))_{i=0}^{j}):= \delta(\zeta|_{[0;j]})$.
	Let ${\A}:= \{A(\zeta|_{[0;j]})\mid \zeta\in\cZ, \tau = \text{ length}(\zeta),j\in [0;\tau-1]\}$. Note that 
	$\A$ is a cover of $Q\backslash T$.
	Now we define a set of sequences in $\A$. Let $\bar{\cR} := \{ (A(\zeta|_{[0;j]}))_{j=0}^{\tau-1}\mid \zeta\in\cZ \wedge \tau = \text{ length}(\zeta) \}$. 
	Then the set $\cR:= \{\alpha T\mid \alpha\in\bar{\cR}\} \cup \{T\}$, which is formed by concatenation of $T$ at the end of the elements in $\bar{\cR}$, satisfies all the three conditions in Definition~\ref{def:reachSpanning} and is a $(Q,T)$-reach-spanning set in $(\A,G)$. 
	Note that $\no Z(\emptyset) = \no\cR_0$ where $\cR_0 := \{\alpha(0)\mid\alpha\in\cR,\alpha(0)\neq T\}$. 
	Let $\hat{\cZ}:= \{ws^\emptyset\mid w\in\cZ$\}, then for any $\zeta\in \hat{\cZ}$ of length $\tau$, a cover sequence $\alpha|_{[0;\tau-2]} = (A(\zeta|_{[0;j]}))_{j=0}^{\tau-2} $ with $\alpha(\tau-1)=T$, we have $\no \hat{P}_\cR ( \alpha|_{[0;t]}) \leq \no Z(\zeta|_{[0;t]}) $ for all $t\in[0;\tau-2]$. Thus $\cN(\cR) \leq R(H)$ which leads to $h(Q,T)\leq R(H) { \leq \inf_{\tilde{H}\in \cH(Q,T)} R(\tilde{H}) + \varepsilon}$. {Since $\varepsilon>0$ is arbitrary, we get $h(Q,T)\leq \inf_{\tilde{H}\in \cH(Q,T)} R(\tilde{H}) $.   }
	

	\emph{Showing $h(Q,T)\geq  {\inf_{H\in \cH(Q,T)} R(H) }$:}
	let $h(Q,T)$ be finite. This implies $(Q,T)$-reachability is satisfiable, thus from Lemma~\ref{lem:QTreachabilityReachSpanningSet} we have the existence of a $(Q,T)$-reach-spanning set $\cR$ in some $(\A,G)$ with {$h(Q,T) + \varepsilon \geq \cN(\cR)$ for some fixed $\varepsilon>0$. We assume $\cR$ to be non-trivial, i.e., $\no \cR >1$.   }
	From $\cR$ we iteratively construct a coder-controller $H = (\bar{S},\gamma,\delta)$ that will also be an element of $\cH(Q,T)$. 
	
	Let at time $t=0$ the state of the system be $x_0\in Q$.
	If $x_0\in T$, then $\gamma(x_0):= s^\emptyset$.
	For $x_0\in Q\backslash T$, we define $s_0 = \gamma(x_0):= A_{0}$ and $\delta(s_0) := G(A_0)$ where $A_0 \ni x_0$ is an element of $\cR_0$ (if multiple such $A_0$, then randomly select one of them). The controller applies the control input $G(A_0) $ to the system. Let $x_1$ denote the next state at time $t=1$. From condition~\ref{def:reachSpanning:successorsCoverPost} in Definition~\ref{def:reachSpanning} we know that $x_1 \in A_1$ for some $A_1 \in P_\cR(A_0)$. 
	We define $s_1 = \gamma(x_0x_1):= A_1$ and $\delta(s_0s_1):= G(A_0 A_1)$ where $A_1 \ni x_1$ and $A_1 \in P_\cR(A_{0})$ (select one, if more than one such $A_1$ ). In this manner, we iteratively define the coder and the controller map $\gamma$ and $\delta$, respectively. 
	Following this scheme, let at time $t$ the state be $x_t$ and let $\alpha\in\cR$ be such that $x_j\in\alpha(j)$ for all $j\in[0;t]$. 
	By definition, every element $\alpha_i$ in $\cR$ is of some finite length $\tau_i$, and $\alpha_i(\tau_i-1)=T$ and $\alpha_i(k)\cap T =\emptyset$, for all $t\in[0;\tau-2]$. Then $x_{\tau-1}\in T$ and we define $\gamma(x_0x_1 \ldots x_{\tau-1}):= s^\emptyset$. No symbol is transmitted at time $\tau-1$. Here, the symbol alphabet is $\bar{S} = S\cup \{s^\emptyset\}$ with ${S} = \{\alpha(t)\mid\alpha\in\cR\} $. This completely defines the coder and the controller map $\gamma$ and $\delta$, respectively. Clearly the coder-controller $H$ is a member of the set $\cH(Q,T)$, and it enforces $(Q,T)$-reachability for $\Sigma$. 
	Let $\mathcal{Z}$ denote the set of all possible finite-length symbol sequences generated in the closed loop by the coder-controller $H$ for initial states in $Q$ and define $\hat{\mathcal{Z}}:= \{\omega s^\emptyset\mid \omega\in \mathcal{Z}\}$.
	Then {$\hat{\cZ} \subseteq \cR $}, with $s^\emptyset := T$.
	Here $Z(\emptyset) \subseteq \cR_0$ and for any $\zeta\in\hat{\cZ}$ of length $\tau$, $Z(\zeta|_{[0;t]}) \subseteq \hat{P}_\cR(\zeta|_{[0;t]})$, for all $t\in[0;\tau-2]$. This gives the data rate $R(H) \leq \cN(\cR) { \leq h(Q,T) + \varepsilon}$. {Since $\varepsilon>0$ is arbitrary, we get $h(Q,T)\geq \inf_{\tilde{H}\in \cH(Q,T)} R(\tilde{H})$}.
\end{proof}

%
Next, we present the relationship between entropies of two systems related with a feedback refinement relation. The result is utilized in the section that follows, on numerical overapproximation to compute an upper bound of the reachability entropy of a continuous-space control system by leveraging its finite abstraction. Particularly,
given a continuous-space control system $\Sigma_1$ which is related under a feedback refinement relation to its finite abstraction $\Sigma_2$, then from Theorem~\ref{thm:FRR} (presented below) we get that an upper bound for the entropy of the abstract system will also be an upper bound for the entropy of the original system.

\subsection{Related systems under feedback refinement relation}\label{sec:frr}
Let us first define the notion of feedback refinement relation introduced in \cite{reissig2016feedback}.
\begin{definition}
	Let $\Sigma_1$ and $\Sigma_2$ be two systems of the form
	\begin{equation}\label{eq:FRRsystems}
		\Sigma_i = (X_i,U_i,F_i) \text{ with } i\in \{1,2\}.
	\end{equation}
	A strict relation $R\subseteq X_1\times X_2$ is a feedback refinement relation from $\Sigma_1$ to $\Sigma_2$ if there exists a map $r:U_2\to U_1$ so that the following inclusion holds for all $(x_1,x_2)\in R$ and $u\in U_2$
	\begin{equation}\label{eq:FRR}
		R(F_1(x_1,r(u))) \subseteq F_2(x_2,u).
	\end{equation}
\end{definition}
Now, we present the main result of this subsection.
\begin{theorem}\label{thm:FRR}
	Consider two systems $\Sigma_i$, $i\in\{1,2\}$ of the form~\eqref{eq:FRRsystems}. Let $T_1\subseteq Q_1\subseteq X_1$ and $T_2\subseteq Q_2\subseteq X_2$ be nonempty sets. Suppose, that $R$ is a feedback refinement relation from $\Sigma_1$ to $\Sigma_2$, {$Q_1 = R^{-1}(Q_2)$}, $T_1 = R^{-1}(T_2)$, $\no R(x_1)=1$ for all $x_1\in X_1$ and $R^{-1}(x_2)\neq \emptyset$ for all $x_2\in Q_2$. Then
	\begin{equation}
		h_1(Q_1,T_1) \leq h_2(Q_2,T_2)
	\end{equation}
	holds.
\end{theorem}
\begin{proof}
	Let $(Q_2,T_2)$-reachability be satisfiable for $\Sigma_2$, then $h_2(Q_2,T_2)$ is finite and from Lemma~\ref{lem:QTreachabilityReachSpanningSet} we have the existence of a $(Q_2,T_2)$-reach-spanning set $\cR_2$ in some $(\A_2,G_2)$ such that {$\cN(\cR_2) \leq h_2(Q_2,T_2) + \varepsilon$ for some fixed $\varepsilon>0$. 
	We assume $\cR_2$ to be non-trivial.     } 
	Since $T_1 = R^{-1}(T_2)$, $\no R(x_1)=1$ for all $x_1\in X_1$ and $R^{-1}(x_2)\neq \emptyset$ for all $x_2\in Q_2$, we have $R(T_1) = T_2$. Thus for any $A_2\subseteq Q_2\backslash T_2$ we have $R^{-1}(A_2) \subseteq Q_1\backslash T_1$.
	Let $\A_1:=\{A_1\subseteq Q_1\backslash T_1 \mid \exists_{A_2\in\A_2} R^{-1}(A_2)=A_1\}$, 
	and for $\alpha_2\in\cR_2$ with $\alpha_1=(R^{-1}(\alpha_2(j)))_{j=0}^{\tau-1}$ we define $G_1(\alpha_1|_{[0;i]}):=r(G_2(\alpha_2|_{[0;i]}))$ for all $i\in[0;\tau-1]$ with $\tau = $ length$(\alpha_2)$. 
	Now we show that the set $\cR_1:= \{ (R^{-1}(\alpha(j)))_{j=0}^{\tau-1}\mid \alpha\in \cR_2\wedge \tau = \text{ length}(\alpha)\} $ is $(Q_1,T_1)$-reach-spanning in $(\A_1,G_1)$. 
	Since $R^{-1}(Q_2)=Q_1$, and $R^{-1}(Q_2\backslash T_2)\subseteq Q_1\backslash T_1$,      we have 
	$R^{-1}(Q_2\backslash T_2) = Q_1\backslash T_1$, and therefore
	$R^{-1}(Q_2\backslash T_2) = R^{-1}(\cup_{A_2\in\A_2}A_2) = $ $\cup_{A_2\in\A_2}R^{-1}(A_2) = $ $\cup_{A_1\in\A_1}A_1 = Q_1\backslash T_1$, thus $\A_1$ is a cover of $Q_1\backslash T_1$. 
	Since $ R^{-1}(T_2) = T_1$ and $R^{-1}(Q_2\backslash T_2)\subseteq Q_1\backslash T_1$ we conclude that the condition~\ref{def:reachSpanning:finalElement} in Definition~\ref{def:reachSpanning} holds for $\cR_1$. Equation~\eqref{eq:FRR} leads to $F_1(x_1,r(u))\subseteq R^{-1}(F_2(x_2,u))$ for all $(x_1,x_2)\in R$ and $u\in U_2$.
	For $\alpha_2\in \cR_2$ and $\alpha_1\in\cR_1$ such that $\alpha_1=(R^{-1}(\alpha_2(j)))_{j=0}^{\tau-1}$, $\tau =$ length$(\alpha_2)$, we observe
	\begin{align}
		F_1(\alpha_1(t),G_1&(\alpha_1|_{[0;t]})) = F_1(R^{-1}(\alpha_2(t)),r(G_2(\alpha_2|_{[0;t]}))) \nonumber \\
		&\subseteq  R^{-1}(F_2(\alpha_2(t),G_2(\alpha_2|_{[0;t]}))) \nonumber \\
		&\subseteq R^{-1} \left(\cup_{A_2\in P_{\cR_2}(\alpha_2|_{[0;t]})} A_2 \right) \label{eq:FRR:pf:RinvCup}  \\
		&= \cup_{A_2\in P_{\cR_2}(\alpha_2|_{[0;t]})} R^{-1}(A_2) \nonumber \\
		&= \cup_{A_1\in P_{\cR_1}(\alpha_1|_{[0;t]})} (A_1). \nonumber
	\end{align}
	The second set inclusion in equation~\eqref{eq:FRR:pf:RinvCup} holds because $\cR_2$ is a $(Q_2,T_2)$-reach-spanning set.
	Thus $\cR_1$ satisfies all the conditions in Definition~\ref{def:reachSpanning}, and is a $(Q_1,T_1)$-reach-spanning set in $(\A_1,G_1)$. Since $\no R(x_1)=1$ for all $x_1\in X_1$ and $R^{-1}(x_2)\neq \emptyset$ for all $x_2\in Q_2$, we have $R^{-1}(A_2)\neq R^{-1}(\tilde{A_2})$ for $A_2\neq \tilde{A_2} \in \A_2$. Thus, for $\alpha_2\in \cR_2$ and $\alpha_1\in\cR_1$ such that $\alpha_1=(R^{-1}(\alpha_2(j)))_{j=0}^{\tau-1}$, $\tau =$ length$(\alpha_2)$, we have $\no\hat{P}_{\cR_2}(\alpha_2|_{[0;t]}) = \no \hat{P}_{\cR_1}(\alpha_1|_{[0;t]})$ for every $t\in[0;\tau-2]$. We also have $\no\cR_{1,0} = \no\cR_{2,0}$ where $\cR_{i,0}:=\{\alpha_i(0)\mid \alpha_i\in\cR_i, \alpha_i(0)\neq T\}$, $i\in\{1,2\}$. Therefore, $\cN(\cR_1) = \cN(\cR_2) {\leq h_2(Q_2,T_2) + \varepsilon}$ and {since $\varepsilon>0$ is arbitrary, we get} $h_1(Q_1,T_1) \leq h_2(Q_2,T_2)$.
	{Note that for the trivial case $\cR_2 = \{T_2\}$, we can use arguments similar to above to show that $\cR_1=\{T_1\}$ will also be a trivially $(Q_1,T_1)$-reach-spanning set.}
\end{proof}

\begin{remark}
Note that assumptions $\no R(x_1)=1$ $\forall x_1\in X_1$ and $R^{-1}(x_2)\neq \emptyset$ $\forall x_2\in Q_2$ in Theorem \ref{thm:FRR} are not restrictive at all and one can always construct finite abstractions as in \cite{reissig2016feedback} satisfying these two conditions. 
\end{remark}
\section{Numerical Overapproximation}\label{sec:numerics}
In this section we describe a procedure for the numerical estimation of $h(Q,T)$ for discrete-time control systems using their finite abstractions.
Consider a system $\Sigma = (X,U,F)$ of the form~\eqref{e:sys} and sets $T\subseteq Q\subseteq X\subseteq \R^n$ and $U\subseteq \R^m$ such that $(Q,T)$-reachability is satisfiable for $\Sigma$. 
We use \texttt{SCOTS} \cite{rungger2016scots} to obtain a controller $\CC:{\A}\to U$ for $(Q,T)$-reachability.
\texttt{SCOTS} is a software tool, written in C++, for automated controller synthesis for nonlinear control systems based on finite abstractions. 
Let $Q\backslash T=\cup_{A\in{\A}}A$, then the domain ${\A}$ of the controller is a partition of $Q\backslash T$. 
Now we describe the construction of a set  $\bar{\cR} = \{\alpha_1, \ldots, \alpha_N\}$ where $N\in \N$, $\alpha_i$ is a sequence of length $\tau_i$, 
$\alpha_i|_{[0;\tau-2]} \in \A^{[0;\tau_i-2]}   $, $\tau_i\in\N$, $1\leq i \leq N$, and $\alpha_i(\tau-1) = T$.
For this, we first  define a set-valued map $D:{\A}\rightrightarrows \A\cup \{T\}$
so that, for each $A\in {\A}$, the set $D(A)$ gives all those elements of ${\A}\cup \{T\}$ that have a nonempty intersection with the image of $A$ with $\CC(A)$ as the control input under the transition function $F$, i.e.,  $D(A):= \{A\in \A\cup \{T\} \mid A\cap F(A,\CC(A))\neq \emptyset\}$. 
Let $\alpha_1(0):= A_1\in {\A}$, $\alpha_1(1):= A_2\in D(\alpha_1(0))$ and so on. 
The sequence $\alpha_1$ terminates when $\alpha_1(j) = T$ and $\tau_1:=$ length$(\alpha_1) = j-1$. 
The set $\bar{\cR}$ is defined to be the collection of all such sequences, i.e., $\bar{\cR}:=\{ (\alpha_i(j))_{j=0}^{\tau_i-1} \mid \alpha_i(0)\in {\A}, \alpha_i(k)\in D(\alpha_i(k-1)), 1 \leq k \leq \tau_i -1, \alpha_i(\tau_i-1) = T \} $.
Let ${\cR}:= \bar{\cR} \cup \{T \}$.
Note that by construction, the controller $\CC$ is such that every state in $Q\backslash T$ will eventually reach $T$ in finite time, while never leaving $Q$. Thus, every $\alpha_i\in\cR$ will have the last element as $T$. 
Next, we define the map $G:\cup_{t\in\Z_{\geq 0}}\A^{[0;t]}\to U$. For $t\in\Z_{\geq 0}$, $\alpha\in \A^{[0;t]}$, define $G(\alpha):= \CC(\alpha(t))$, if $\alpha(t)\in {\A}$, and $G(\alpha):= u$, for some fixed $u\in U$ when $\alpha(t) = T$.   
Note that $\cR$ satisfies all the three conditions in Definition~\ref{def:reachSpanning}, thus it is a $(Q,T)$-reach-spanning set, and thus $h(Q,T)\leq \cN(\cR)$.
We describe the procedure to compute $\cN(\cR)$ in the next section via Example~\ref{eg:rommtemp}.


\section{Examples}\label{sec:Examples}

In this section we present three examples. For the  system with discrete $X$ and $U$ in Example~\ref{eg:discrete}, we compute the value of the reachability entropy and also a $(Q,T)$-admissible coder-controller with data rate equal to the entropy. 
For the scalar linear system in Example~\ref{eg:scalar} with a stable eigenvalue, we show that a non-zero bit rate is required for $(Q,T)$-reachability. Finally, for the case of room temperature control in a circular building in Example~\ref{eg:rommtemp}, we elaborate on the numerical procedure to compute an upper bound of the reachability entropy. 

\begin{example}\label{eg:discrete}
	Consider an instance of~\eqref{e:sys} with $U=\{a,b\}$, $X=\{0,1,2,3\}$ and $F$ illustrated by the following state diagram.
	
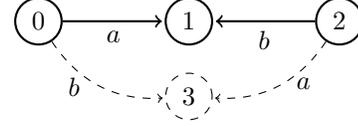
\begin{center} 
		\begin{tikzpicture}[->,thick,shorten >=1pt]
			\tikzstyle{state} = [draw, circle, minimum height=1.5em, minimum width=1.5em]
			
			\node[state] (0) at (-2,0) {$0$};
			\node[state] (1) at (0,0) {$1$};
			\node[state] (2) at (2,0) {$2$};
			\node[thin,dashed,state] (3) at (0,-1) {$3$};
		
      \path (0) edge  node[below] {$a$} (1);

      \path (2) edge  node[below] {$b$} (1);


      \path[thin,dashed] (0) edge[bend right]  node[near start,below] {$b$} (3);
      \path[thin,dashed] (2) edge[bend left ]  node[near start,below] {$a$} (3);
		\end{tikzpicture}
  \end{center}
  Let $Q=\{0,1,2\}$ and $T=\{1\}$. The transitions that lead outside $Q$ and the states that are outside $Q$ are marked by dashed lines. 
  
  First, we compute $h(Q,T)$ through construction of a $(Q,T)$-reach-spanning set.
  Consider a cover $\A:=\{A_1=\{0\}, A_2 = \{2\}\}$ of $Q\backslash T$ and for $t\in \Z_{\geq 0}$, $\alpha\in \A^{[0;t]}$, we define a map $G$ as $G(\alpha):= G(\alpha(t))$, $G(A_1):= a$ and  $G(A_2):= b$. 
  Further consider the sets $\cR:=\{A_1T, A_2T, T\}$ and $\cR_0:=\{A_1,A_2\}$. The set $\cR$ satisfies all the conditions in Definition~\ref{def:reachSpanning} and thus it is $(Q,T)$-reach-spanning in~$(\A,G)$. From equations~\eqref{eq:Balpha} and~\eqref{eq:NofR} we get
  $\cB(\A_1T)=\cB(\A_2T) = 1 $, $\cB(T)=0$ and $\cN(\cR)=1$.
  Since $\cR$ is the only $(Q,T)$-reach-spanning set in $(\A,G)$ and there does not exist any other values of the tuple $(\A,G)$ with a different $(Q,T)$-reach-spanning set, we obtain $h(Q,T)= 1$. 

  Next, we construct a coder-controller $H=(\bar{S},\gamma,\delta)\in \cH(Q,T)$. 
  Let $\bar{S}= S\cup \{s^\emptyset\}$, $S=\{0,2\}  $, $\gamma(0):=0$, $\delta(0):= a$, $\gamma(2):= 2$, $\delta(2) := b$, $\gamma(1):=s^\emptyset$, $\cZ:=\{0,2\}$, $\cZ(\emptyset)=\cZ$, and $\hat{\cZ}:=\{0s^\emptyset, 2s^\emptyset\}$. 
  The coder-controller $H$ is $(Q,T)$-admissible. Then, from~\eqref{eq:datarate} we obtain $R(H)=1$. \qed
\end{example}





\begin{example}\label{eg:scalar}
	Consider an instance of \eqref{e:sys} with $X=\R$, {$U=\{-0.5,0.75\}$} and 
    $F(x,u) = 0.5 x + u,$
	with a safe set {$Q=[0,1.4]\cup[2,6]$, and a target set $T=[0,1.4]$}. There is no single value of control input $u\in U$ that can enforce the $(Q,T)$-reachability.
	The set $Q\backslash T$ needs a cover with at least two members, each of which is assigned a distinct control input. For example, consider a cover $\A=\{A_1,A_2\}$ where {$A_1 = [3.75,6]$ is assigned the control input $u_1=0.75$, and $A_2=[2,3.75]$ is assigned the control $u_2=-0.5$}. Then at time $t=0$ the coder must transmit one bit to distinguish between $A_1$ and $A_2$. Thus
	a {nonzero} bit rate is required to enforce $(Q,T)$-reachability. 
	\qed

\end{example}

Next, we numerically compute an upper bound of the reachability entropy for the problem of temperature regulation in a circular building consisting of $3$ rooms, each equipped with a heater. 

\begin{example}\label{eg:rommtemp}
	The temperature $\T_i$ of the room $i\in\{1,2,3\}$ is described by the following difference equation borrowed from~\cite{meyer2017compositional} 
	\begin{multline*}
		F_i(\T,u_i) = \T_i + \alpha(\T_{i+1}+\T_{i-1}-2\T_i) \\+ \beta(\T_e-\T_i) + \gamma(\T_h-\T_i)u_i,
	\end{multline*}
	where $\T_{i+1}$ and $\T_{i-1}$ are the temperature of the neighbour rooms (with $\T_0=\T_3$ and $\T_4=\T_1$), $\T_e=-1^o C$ is the outside temperature, $\T_h=50^oC$ is the heater temperature, $u_i\in[0,0.6]$ is the control input for room $i$ and the conduction factors are given by $\alpha=0.45$, $\beta=0.045$, and $\gamma=0.09$. The temperature is desired to be taken into the set $T=[22,24]^3$, while never leaving $Q=[17.4,24]^3$.

	With $\eta_s=[1.2\quad 1.2\quad 1.2]^T$ and $\eta_i=[0.01\quad 0.01\quad 0.01]^T$ as the state and input grid parameters, respectively, for tool \texttt{SCOTS}, we compute a controller $\CC:\cB \to U$ for $(Q,T)$-reachability, i.e., it forces the closed-loop system to reach $T$ in finite time from any state in $Q$, while never leaving $Q$. 
	The domain $\cB$ of the controller is a set containing $215
	$ state grid cells. 
	To obtain a smaller upper bound, it is desirable that the partition has lower cardinality. We use \texttt{dtControl} \cite{ashok2020dtcontrol} to group cells together, which have identical control inputs assigned to them under the map $\CC$.
	\texttt{dtControl} is a software tool, written in Python, for transforming memoryless symbolic controllers into various compact and more interpretable representations.
	The use of \texttt{dtControl} gives a coarse partition ${\A}$ with $72$ elements, and a map $\bar{\CC}:{\A}\to U$. 
	Note that for  $B\in \cB$ and $A\in{\A}$ with $B\subset A$, the sets $B$ and $A$ have the same controls assigned to them, i.e., $\CC(B) = \bar{\CC}(A)$.
	To obtain a $(Q,T)$-reach-spanning set, we now construct a directed weighted graph $\cG$, with ${\A}\cup \{T\}$ as the set of nodes. 
	For $A_1, A_2\in {\A}$, there is an edge from $A_1$ to $A_2$ if $F(A_1,\bar{\CC}(A_1))\cap A_2\neq \emptyset$. 
	All the outgoing edges for any node are assigned the same edge weight equal to the base-$2$ logarithm of the number of outgoing edges for the node.
	The set of sequences in ${\A}$, generated by traversing paths in the graph that start from any node and terminate at $T$, gives a set ${\cR}$ as described in Section~\ref{sec:numerics}. The set ${\cR}$ is a $(Q,T)$-reach-spanning set.
	For each path $\alpha$ (of length $\tau$) that terminate at $T$, we compute $\cB(\alpha)$ which is	defined in~\eqref{eq:Balpha}
	by 
	\begin{equation*}
		\cB(\alpha) = \frac{1}{\tau-1}[\sum_{t=0}^{\tau-3}w(\alpha(t),\alpha(t+1)) + \log_2\no \A],
	\end{equation*}
	where $w(\alpha(t),\alpha(t+1))$ is the weight of the edge from the node $\alpha(t)$ to the node $\alpha(t+1)$.
	Here, we obtain $\cN(\cR) = \max_{\alpha\in \cR}\cB(\alpha) = 6.1699 \geq h(Q,T)$. \qed

\end{example}


\section{Conclusion}
In this work, we introduced for the first time a new notion of reachability entropy defined in terms of the sequences of cover elements. The entropy is shown to be equal to the minimum data rate amongst coder-controllers that can enforce the rechability specification. 
We also described the relation between entropies of two systems related under a feedback refinement relation. Further, we provided a graph based numerical procedure for the estimation of the reachability entropy. Finally we illustrated the effectiveness of the results via three examples.








\bibliographystyle{IEEEtran}
\bibliography{literatur.bib}

\end{document}